\documentclass[10pt,a4paper]{article}
\usepackage[utf8]{inputenc}
\usepackage[english]{babel}
\usepackage[T1]{fontenc}
\usepackage{amsmath}
\usepackage{amsfonts}
\usepackage{amssymb}
\usepackage[left=3cm,right=3cm,top=4cm,bottom=3cm]{geometry}
\usepackage{graphicx}
\usepackage{caption} 
\usepackage{ulem}
\usepackage{comment}
\usepackage{color}
\usepackage{pgf,tikz}
\usetikzlibrary{matrix,arrows}

\usepackage{charter}
\usepackage{layout}

\usepackage{graphicx}
\usepackage[all]{xy}

\usepackage{amsthm}
\usepackage{amsmath}
\usepackage{amssymb}
\usepackage{mathrsfs}

\newtheorem{theo}{Theorem}
\newtheorem{proposition}[theo]{Proposition}

\newtheorem*{definition}{Definition}
\newtheorem*{remarque}{Remark}

\newtheorem*{definition*}{Definition}
\newtheorem{theo*}{Theorem}[section]
\newtheorem{question}{Question}
\newtheorem{proposition*}[theo*]{Proposition}
\newtheorem{lemme*}[theo*]{Lemma}

\newcommand{\pp}{\text{.}}
\newcommand{\vv}{\text{,}}
\newcommand{\dt}{\partial_t}

\DeclareMathOperator{\supp}{supp}

\def\Z{\mathbf{Z}}

\def\R{\mathbf{R}}

\begin{document}

\begin{center}{\Large{\textbf{Spectral asymptotics for the vectorial damped wave equation }}}\end{center}

\bigskip

\begin{center}
Guillaume Klein
\end{center}

\bigskip	

\begin{abstract}The eigenfrequencies associated to a scalar damped wave equation are known to belong to a band parallel to the real axis. In \cite{sjo00} J. Sjöstrand showed that up to a set of density $0$, the eigenfrequencies are confined in a thinner band determined by the Birkhoff limits of the damping term. In this article we show that this result is still true for a vectorial damped wave equation. In this setting the Lyapunov exponents of the cocycle given by the damping term play the role of the Birkhoff limits of the scalar setting. 
\end{abstract}

\section{Introduction}

\subsection{Setting}
\paragraph*{}Let $(M,g)$ be a compact Riemannian manifold of dimension $d$ without boundary, we will moreover assume that $(M,g)$ is connected and $C^\infty$. Let $\Delta$  be the Laplace-Beltrami operator on $M$ for the metric $g$ and let $a$ be a $C^\infty$-smooth function from $M$ to $\mathscr{H}_n(\mathbf C)$, the space of hermitian matrices of dimension $n$. We are interested in the following system of equations
\begin{equation}\label{dampedwaveequation''}
\left\lbrace
\begin{array}{l}
 (\partial_t^2 -\Delta +2a(x)\partial_t)u=0 \; \text{ in } \; \mathcal D'(\mathbf R\times M)^n \\

u_{|t=0}=u_0\in H^1(M)^n\; \text{ and } \; \partial_tu_{|t=0}=u_1\in L^2(M)^n\pp
\end{array}
\right.
\end{equation}
Let $H=H^1(M)^n\oplus L^2(M)^n$ and define on $H$ the unbounded operator
\[
A_a=\begin{pmatrix}
0 & \mathrm{Id}_n \\
\Delta & -2a
\end{pmatrix} \text{ of domain } D(A_a)=H^2(M)^n\oplus H^1(M)^n\pp
\]
By application of Hille-Yosida theorem to $A_a$ the system \eqref{dampedwaveequation''} has a unique solution in the space $ C^0(\mathbf R,H^1(M)^n)\cap C^1(\mathbf R, L^2(M)^n)$. If $u$ is a solution of \eqref{dampedwaveequation''} then we can define the energy of $u$ at time $t$ by the formula
\begin{equation}\label{eq:energie}
E(u,t)=\frac{1}{2}\int_M |\dt u(t,x)|^2+|\nabla u(t,x)|^2_g d x
\end{equation}
where $|v|^2_g=g_x(v,v)$. An integration by parts shows that
\begin{equation}\label{eq:formuleenergie''}
\frac{d}{d t}E(u,t)=-\int_M \left\langle 2a(x)\dt u(t,x), \dt u(t,x) \right\rangle_{\mathbf C^n} d x\vv
\end{equation}
so if $a$ is not null then the energy is not constant. In particular if $a$ is positive semi-definite the energy is non-increasing and $a$ can be seen as a dampener. In the problem of stabilisation of the wave equation one is interested with the long time behavior of the energy of solutions to \eqref{dampedwaveequation''}. The spectrum of $A_a$ is obviously related to this long time behavior although it is not sufficient to completely describe it (see for example \cite{leb93}). In this article we will thus be interested in the asymptotic distribution of the eigenvalues of $A_a$. Besides the possible applications to stabilisation the problem of determining an asymptotic distribution of eigenvalues for a non-self-adjoint operator is an interesting problem by itself but it can also be seen as one of the many generalization of Weyl's law. This problem has already been studied in the scalar case ($n=1$) for example in \cite{sjo00}. The aim of this article is to expend some results of \cite{sjo00}  to the case where  $n\geq 1$.

\subsection{Results}
\paragraph*{}It is a classical fact that the spectrum of $A_a$ is symmetric with respect to the real axis and that it contains only discrete eigenvalues with finite multiplicities, this is proved using Fredholm theory. If $u$ is a stationary solution of \eqref{dampedwaveequation''} we can write $u(t,x)=e^{it\tau}v(x)$ and we are lead to the equation
\begin{equation}\label{dampedwaveequation2}
(-\Delta - \tau^2 +2ia\tau)v=0 \pp
\end{equation}
We say that $\tau\in \mathbf C$ is an eigenvalue for \eqref{dampedwaveequation2} if there is a non-zero solution $v$ of \eqref{dampedwaveequation2}. Note that $\tau$ is an eigenvalue of \eqref{dampedwaveequation2} if and only if $i\tau$ is an eigenvalue of $A_a$.

\begin{definition}
We note $\mathrm{sp}(a(x))$ the set of all the eigenvalues of $a(x)$ and define
\[
a^-=\inf_{x\in M} \min \mathrm{sp}(a(x)) \;\text{ and }\; a^+=\sup_{x\in M} \max \mathrm{sp}(a(x))\pp
\]

\end{definition}

 By reasoning on the energy with \eqref{eq:formuleenergie''} it is easy to see that $\mathfrak{Im}(\tau)$ can only be in the interval $[2a^-;2a^{+}]$. We are now interested in the asymptotic repartition of $\mathfrak{Im}(\tau)$ when $|\tau|$ goes to infinity. Since the spectrum of $A_a$ is invariant under complex conjugation we will only be interested in the limit $\mathfrak{Re }(\tau)\to +\infty$.

\paragraph*{} We define the function $p: T^*M \to \mathbf R$ by $p(x,\xi)=g^x(\xi,\xi)$, notice that $p$ is the principal symbol of $-\Delta$. We also define $\phi$ as the Hamiltonian flow generated by $p$, it can also be interpreted as the geodesic flow on $T^* M $ travelled at twice the speed. In what follows $(x_0,\xi_0)$ will be a point of $T^*M$ and we will write $(x_t,\xi_t)$ for $\phi_t(x_0,\xi_0)$.
\begin{definition}
Let $t$ be a real number, we define the function $G_t : p^{-1}(1/2) \to \mathscr M_n(\mathbf C)$ as the solution of the differential equation
\begin{equation}\label{eq:equationG'}
\left\lbrace
\begin{array}{l}
G_0(x_0,\xi_0)=\mathrm{Id}_n \\
\partial_t G_t(x_0,\xi_0)=-a(x_t)G_t(x_0,\xi_0)\pp
\end{array}
\right.
\end{equation}
This definition naturally extends to any point $(x_0,\xi_0)\in T^*M$.
\end{definition}
Notice that $G$ is a cocycle map, this means that for every $s,t\in \mathbf R$ and every $(x,\xi)\in T^*M$ we have the equality $G_{t+s}(x,\xi)=G_t(\phi_s(x,\xi))G_s(x,\xi)$. When $n=1$ the function $a$ is real valued and we simply have
\begin{equation}\label{eq:equationGscalaire}
G_t(x_0,\xi_0)=\exp\left(-\int_0^t a(x_s) d s\right)\pp
\end{equation}
This formula however ceases to be correct when $n$ is greater than $1$. Throughout the entire article if $E$ is a vector space and $\|\cdot\|_*$ is a norm on $E$ we will also write $\|\cdot\|_*$ for the associated operator norm on $\mathcal{L}(E)$. 
\begin{definition*}
For every positive time $t$ we define the following quantities 
\[C^-_t =\frac{-1}{t}\sup_{(x_0,\xi_0)\in p^{-1}(1/2)} \ln \left( \|G_t(x_0;\xi_0)\|_2 \right)\]
\[\text{ and } C^+_t =\frac{-1}{t}\inf_{(x_0,\xi_0)\in p^{-1}(1/2)} \ln \left( \|G_t(x_0;\xi_0)^{-1}\|_2^{-1} \right)\pp\]
We will also note $C_\infty^\pm = \lim_{t\to +\infty}C_t^\pm$.
\end{definition*}
It is easy to show that this limit always exists using a sub-additivity argument. Note that the existence and the value of the limit does not depend on the choice of the norm since they are all equivalent, however it is sometimes easier to work with the operator norm associated with the euclidian norm on $\mathbf C^n$. Again, if $n=1$ we get the simpler expression 
\[
C^-_t=\inf_{(x,\xi)\in p^{-1}(1/2)} \frac{1}{t}\int_0^t a(x_s) ds \;\text{ and }\; C^+_t=\sup_{(x,\xi)\in p^{-1}(1/2)} \frac{1}{t}\int_0^t a(x_s) ds\pp
\]
For an analogue to this expressions when $n\geq 1 $ see \cite{thesis} .

\begin{theo*}\label{th:1}
For every $\varepsilon>0$ there is only a finite number of eigenvalues of \eqref{dampedwaveequation2} outside the strip $\mathbf R + i ]C_\infty^--\varepsilon; C_\infty^+ + \varepsilon[$. 
\end{theo*}

In the scalar case this was first proved by Lebeau \cite{leb93} using microlocal defect measures. In fact, because of the setting of his article, he only proved it for the upper bound and $a\leq 0$ but his proof easily extends to Theorem \ref{th:1} with $n=1$. Theorem \ref{th:1} was also proved by Sjöstrand in \cite{sjo00} for $n=1$ using different techniques. The argument of Sjöstrand relies on a conjugation by pseudo-differential operators to replace the damping term $a$ by its average on geodesics of length $T$. Because of the non commutativity of matrices it seems that this argument cannot be modified in a straight forward manner to prove the vectorial case ($n\geq1$). In \cite{kle17}, using the same technique as Lebeau, the upper bound of Theorem \ref{th:1} was proved in the general case $n\geq 1$ for a function $a$ valued in $\mathscr H_n^+(\mathbf C)$, the space of Hermitian positive semi-definite matrices. Once again the argument used there can easily be adapted to prove Theorem \ref{th:1} in its full generality (see \cite{thesis}). 

This result is the best possible in the sense that there can be infinitely many eigenvalues of \eqref{dampedwaveequation2} outside of the strip $\mathbf R + i [C_\infty^-; C_\infty^+]$. However we are going to show that in a way, ``most'' of the eigenvalues of \eqref{dampedwaveequation2} are in fact in a narrower strip. In order to give a meaning to the previous statement we need the following result.

\begin{theo*}\label{th:2}
The number of eigenvalues $\tau$ with $\mathfrak{Re}(\tau) \in [0;\lambda]$ is equivalent to 
\[
n\left(\frac{\lambda}{2\pi}\right)^d \iint_{p^{-1}([0;1])} 1 dxd\xi
\]
when $\lambda$ goes to $+\infty$. Moreover the remainder is a $\mathcal{O}(\lambda^{d-1})$.
\end{theo*}

For $n=1$ this result was first proved in \cite{mama82} by Markus and Matsaev and then independently in \cite{sjo00} by Sjöstrand. Once again the proof can easily be adapted to our case $n\geq 1 $.

\paragraph*{}We now want to make further estimations of the asymptotics of the imaginary part of the eigenvalues of \eqref{dampedwaveequation2}. Recall that the matrix $G$ is a cocycle and that the geodesic flow on $p^{-1}(1/2)$ preserves Liouville's measure. Thus we can use Kingman's subadditive ergodic theorem to show that the limits  
\[
\lambda_n(x,\xi)=\lim_{t\to\infty}\frac{1}{t} \log\left\|G_t(x,\xi)\right\|_2 \;\text{ and }\; \lambda_1(x,\xi)=\lim_{t\to \infty} \frac 1 {t} \log\left( \left\|G_t(x,\xi)^{-1}\right\|_2^{-1}\right)
\] 
exist for almost every $(x,\xi)\in p^{-1}(1/2)$. Moreover the functions $(x,\xi)\mapsto \lambda_1(x,\xi)$ and $(x,\xi)\mapsto \lambda_n(x,\xi)$ are both measurable and bounded. Note that the existence and the value of the limit does not depend on the choice of the norm because they are all equivalent. Note also that $\lambda_1$ and $\lambda_n$ are respectively the smallest and largest Lyapunov exponents defined by the multiplicative ergodic theorem of Oseledets. The statement of Oseledets theorem can be found in Annex \ref{an:oseledets}. We now define 
\[
\Lambda^- = \mathrm{ess} \inf \lambda_1(x,\xi) \text{ and } \Lambda^+=\mathrm{ess} \sup \lambda_n(x,\xi)\pp
\]
In the scalar case $n=1$ so obviously $\lambda_1=\lambda_n$ and we have
\[
\Lambda^+=-\underset{(x,\xi)\in p^{-1}(1/2)}{\mathrm{ess \; inf}} \lim_{t\to \infty}  \frac{1}{t}\int_0^t a(x_s) ds \; \text{ and } \; \Lambda^-=-\underset{(x,\xi)\in p^{-1}(1/2)}{\mathrm{ess \; sup}} \lim_{t\to \infty} \frac 1 t  \int_0^t a(x_s) ds\pp
\] 
Notice the minus sign in comparison to the definition of $C_\infty^\pm$. In general we have 
\[
C_\infty^- \leq -\Lambda^+ \leq -\Lambda^- \leq C_\infty^+
\]
and every inequality is usually strict. We are  now ready to state the main result of this article.

\begin{theo*}\label{th:3}
For every $\varepsilon>0$ the number of eigenvalues $\tau$ satisfying $\mathfrak{Re}(\tau)\in [\lambda ; \lambda+1]$ and $\mathfrak{Im}(\tau)\notin ]-\Lambda^+ -\varepsilon ; -\Lambda^-+\varepsilon[$ is $o(\lambda^{d-1})$ when $\lambda$ tends to infinity.
\end{theo*}

In view of Theorem \ref{th:2} and Theorem \ref{th:3} we see that, up to a null density subset, all of the eigenvalues of \eqref{dampedwaveequation2} have their imaginary part in the interval $]-\Lambda^+ - \varepsilon; -\Lambda^-+\varepsilon[$. Theorem \ref{th:3} was proved by Sjöstrand \cite{sjo00} when $n=1$ and the asymptotics was then refined by Anantharaman for a negatively curved manifold in \cite{ana10}. The main goal of this article is to prove Theorem \ref{th:3} in the general case $n\geq 1$. As for Theorem \ref{th:1} the arguments used by Sjöstrand in \cite{sjo00} seems not to work anymore when $n\geq 1$. This mainly comes from the fact that matrices do not commute and thus formula \eqref{eq:equationGscalaire} is no longer true when $n\geq 1$. 

\subsection{Two open questions}
If we take $n=1$ then there is only one Lyapunov exponent for $G$ which is simply the opposite of the Birkhoff average of $a$ :
\[\lambda_1(x_0,\xi_0)=\lim_{t\to \infty} \frac{-1}{t}\int_0^t a(x_s) ds \pp\]
If we make the assumption that the geodesic flow is ergodic on $M$ we get
\[
\Lambda^-=\Lambda^+=\frac{-1}{\mathrm{vol}(M)}\int_M a(x)dx = \lambda_1(x,\xi) \;\text{ a.e.}
\]
and Theorem \ref{th:3} tells us that most of the eigenvalues of $A_a$ are concentrated around the vertical line of imaginary part $\frac{-1}{\mathrm{vol}(M)}\int_M a(x) dx$. Now if we drop the assumption $n=1$ and keep the ergodic assumption we do not necessarily have $\Lambda^+=\Lambda^-$ but the Lyapunov exponents $\lambda_i$ defined by Theorem \ref{th:oseledets} will be constant almost everywhere. We write $\lambda_i$ for the almost sure value of the function $(x,\xi) \mapsto \lambda_i(x,\xi)$ and we thus have $\lambda_1=\Lambda^-$ and $\lambda_n=\Lambda^+$. Theorem \ref{th:3} tells us that most of the eigenvalues of $A_a$ will be concentrated around the strip $\{z\in \mathbf C : \mathfrak{Im}(z)\in[-\lambda_n;-\lambda_1]\}$ but it seems natural to ask if the following stronger property holds.
\begin{question}
Is it true that for every $\varepsilon>0$ the number of eigenvalues $\tau$ satisfying $\mathfrak{Re}(\tau)\in [\lambda ; \lambda+1]$ and \[\mathfrak{Im}(\tau)\notin \bigcup_{i=1}^n]\lambda_i -\varepsilon ; \lambda_i+\varepsilon[\] is a $o(\lambda^{d-1})$ when $\lambda$ tends to infinity ?
\end{question}
It seems that at the present time, the techniques developed in this article do not allow us to answer this question. It was however proved in \cite{thesis} that the answer is yes when the manifold $M$ is just a circle. The proof for this result relies on microlocal deffect measures and the fact that on the circle the Lyapunov exponents are constant everywhere and not almost everywhere.

\begin{question}
If the answer to the first question is yes, is it true that the eigenvalues are equally distributed between the Lyapunov exponents ? 
\end{question}

Let us rephrase this more precisely. Let $\lambda_1 \leq \ldots \leq \lambda_n$ be the constant almost everywhere Lyapunov exponents of $G$. We want to know if for $\varepsilon	>0$ small enough and for every $i\in \{1,\ldots, n \}$ the number of eigenvalues of \eqref{dampedwaveequation2} in the box $[0;\lambda]+i]\lambda_i-\varepsilon ; \lambda_i + \varepsilon[$ is equivalent to
\[ k_i \left(\frac{\lambda}{2\pi}\right)^d \iint_{p^{-1}([0;1])} 1  dx d \xi
\]
when $\lambda $ goes to $+\infty$ and where $k_i$ is the multiplicity of the Lyapunov exponent $\lambda_i$.

Unfortunately the techniques used to answer Question 1 for $M= \R/2\pi\Z$ are not powerful enough to answer Question 2 even in that simple setting. 

\subsection{Plan of the article}
\paragraph*{}Section \ref{section:preuvetheorem3} is dedicated to the proof of Theorem \ref{th:3} which starts by a semi-classical reduction. The general idea of the proof is to express the eigenvalues of \eqref{dampedwaveequation2} as zeros of some Fredholm determinant depending holomorphically in $z$ and then to use Jensen's formula to bound the number of these zeros. 

In order to construct the aforementioned Fredholm determinant and to get the appropriate bound we need to construct some approximate resolvent for $-\Delta - \tau^2 +2ia(x)\tau$, this is the object of Proposition \ref{prop:ResolvanteApprochee}. The proof of Proposition \ref{prop:ResolvanteApprochee} is postponed to Section \ref{section:preuvepropresolvanteapprochee} and represents the core of this article. Section \ref{section:preuvepropresolvanteapprochee} starts by a sketch of the proof of Proposition \ref{prop:ResolvanteApprochee} for easier understanding.

The article then ends with two annexes. The first one presents the semi-classical anti-Wick quantization and its basic properties. The second annex presents the multiplicative ergodic theorem of Oseledets.

\subsection*{Acknowledgments}
This article is for the vast majority a reproduction of the last chapter of my PhD Thesis (see \cite{thesis}). I would like to thank again my advisor, Nalini Anantharaman. There is no doubt that this article would not have seen the light of day without her help and support. 

\section{Proof of Theorem \ref{th:3}}\label{section:preuvetheorem3}
\paragraph*{} The first step of the proof is to perform a semi-classical reduction borrowed from \cite{sjo00}. Recall from the Introduction that $i\tau$ is an eigenvalue of $A_a$ if and only if there exists some non zero $v:M \to \mathbf C^n$ such that 
\[
(-\Delta -\tau^2+2ia\tau)v=0\pp
\]
We are interested in the asymptotic behaviour of the eigenvalues of $A_a$ and since its spectrum is invariant by complex conjugation we can restrict ourself to the case $\mathfrak{Re}(\tau) \to +\infty$. Let us call $h$ our semiclassical parameter tending to zero and let $i\tau$ be an eigenvalue of $A_a$, depending on $h$, such that $h\tau=1+o(1)$ when $h$ goes to zero. If we write $\tau=\kappa/h$ the previous equation becomes
\[
(-h^2\Delta-\kappa^2+2ia\kappa h)v=0\pp
\]
Now if we write $z=\kappa^2$, and $\kappa=\sqrt{z}$ with $\mathfrak{Re}(z)>0$ the equation becomes
\[
(-h^2\Delta+2iha\sqrt{z}-z)v=0\pp
\]
We might finally rewrite it as 
\begin{equation}
(\mathcal P -z ) v=0
\end{equation}
with $\mathcal{P}=\mathcal{P}(z)=P+ihQ(z)$, where $P=-h^2\Delta$ is the semiclassical Laplacian and $Q(z)=2a\sqrt{z}$. Note that $P$ is self adjoint, $Q$ depends holomorphically on $z$ in a neighbourhood of $1$ and it is self adjoint whenever $z$ is a positive real number. Throughout the rest of the article we will use differential operators depending on the semi-classical  parameter $h$, an exposition of the theory of $h$-pseudo-differential operators is given for example in \cite{zwo12}.

\begin{remarque}
Notice that $z=\kappa^2$ and that $(1+x)^2=1+2x+o(x)$ so we have \[h^{-1}\mathfrak{Im}(z) = h^{-1}2\mathfrak{Im}(\kappa)+o(1) = 2\mathfrak{Im}(\tau)+o(1)\pp\]
This explains the appearance of some multiplications by two in the rest of the article.
\end{remarque}

\paragraph*{} According to this semi-classical reduction, finding an upper bound on the number of eigenvalues of $\mathcal{P}$ in an open set $1+h\widetilde\Omega$ yields an upper bound on the number of eigenvalues of  \eqref{dampedwaveequation2} in an open set $h^{-1}+\widetilde \Omega/2+o(1)$.

\begin{definition*}
Let $\varepsilon>0$ be fixed, we then put
\[
\widetilde\Omega=\{z\in \mathbf C : \mathfrak{Re}(z)\in ]-2;2[, \; \mathfrak{Im}(z)\in ]2a^- -3;-\Lambda^+-\varepsilon/2[\}
\] 
\[
\widetilde\omega=\{z\in \mathbf C : \mathfrak{Re}(z)\in ]-1;1[, \; \mathfrak{Im}(z)\in ]2a^- -2;-\Lambda^+-\varepsilon[\}
\]
\[
\widetilde{z_0}=i(2a^--1)\pp
\]

\end{definition*}

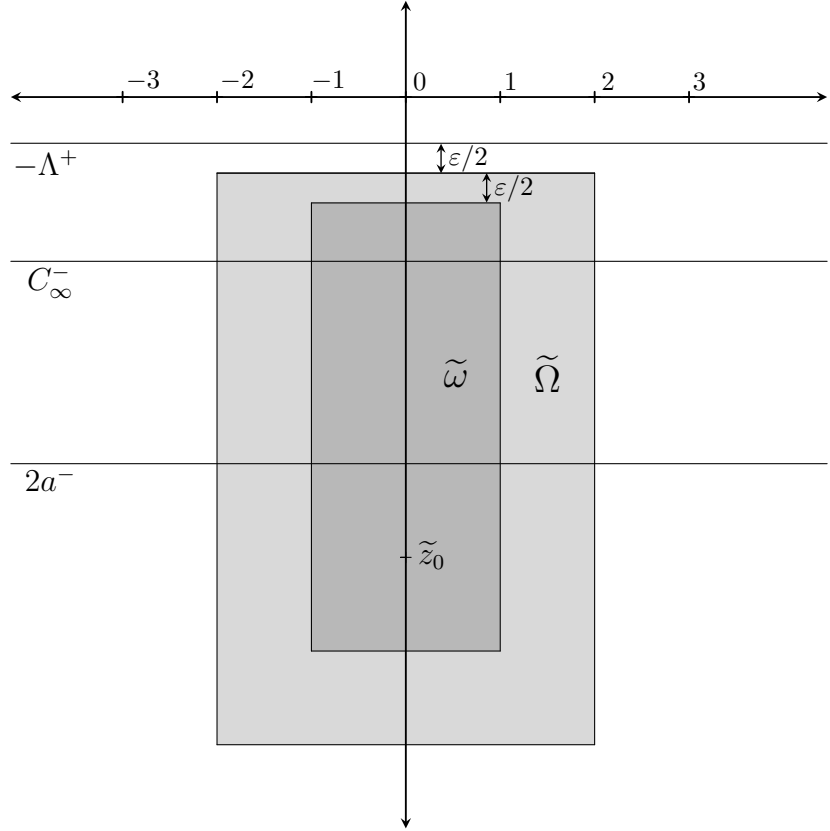
\begin{figure}[ht]
\definecolor{uuuuuu}{rgb}{0.26666666666666666,0.26666666666666666,0.26666666666666666}
\begin{center}
\begin{tikzpicture}[line cap=round,line join=round,>=triangle 45,x=0.4cm,y=0.4cm]
\clip(-7.652192721672691,-22.875768698947027) rectangle (19.200530641808154,4.559022367499436);
\fill[fill=black,fill opacity=0.15000000596046448] (-0.8700322697098821,-1.1528803749783612) -- (11.55176600808043,-1.1528803749783612) -- (11.55176600808043,-20.10109733567344) -- (-0.8700322697098821,-20.10109733567344) -- cycle;
\fill[fill=black,fill opacity=0.15000000596046448] (2.235417299737696,-2.141866847318108) -- (8.446316438632852,-2.141866847318108) -- (8.446316438632852,-16.995647766225858) -- (2.235417299737696,-16.995647766225858) -- cycle;
\draw [<->][>=stealth][line width=0.7pt,domain=-7.652192721672691:19.200530641808154] plot(\x,{(--0.7362204647058319-0.*\x)/0.5382779253709131});
\draw [<->][>=stealth][line width=0.7pt] (5.340866869185273,-22.875768698947027) -- (5.340866869185273,4.559022367499436);
\draw [domain=-7.652192721672691:19.200530641808154] plot(\x,{(--0.08822046989325452-0.*\x)/-0.5382779253709131});
\draw [domain=-7.652192721672691:19.200530641808154] plot(\x,{(--2.1944731818991396-0.*\x)/-0.5382779253709131});
\draw [domain=-7.652192721672691:19.200530641808154] plot(\x,{(--5.8051921167663725-0.*\x)/-0.5382779253709131});
\draw (-0.8700322697098821,-1.1528803749783612)-- (11.55176600808043,-1.1528803749783612);
\draw (-0.8700322697098821,-1.1528803749783612)-- (11.55176600808043,-1.1528803749783612);
\draw (11.55176600808043,-1.1528803749783612)-- (11.55176600808043,-20.10109733567344);
\draw (11.55176600808043,-20.10109733567344)-- (-0.8700322697098821,-20.10109733567344);
\draw (-0.8700322697098821,-20.10109733567344)-- (-0.8700322697098821,-1.1528803749783612);
\draw (2.235417299737696,-2.141866847318108)-- (8.446316438632852,-2.141866847318108);
\draw (8.446316438632852,-2.141866847318108)-- (8.446316438632852,-16.995647766225858);
\draw (8.446316438632852,-16.995647766225858)-- (2.235417299737696,-16.995647766225858);
\draw (2.235417299737696,-16.995647766225858)-- (2.235417299737696,-2.141866847318108);
\begin{scriptsize}
\draw[color=black] (5.8,1.9) node {\normalsize $0$};
\draw[color=black] (2.8,1.9) node {\normalsize $-1$};
\draw[color=black] (-0.2,1.9) node {\normalsize $-2$};
\draw[color=black] (8.8,1.9) node {\normalsize $1$};
\draw[color=black] (12,1.9) node {\normalsize $2$};
\draw[color=black] (15,1.9) node {\normalsize $3$};
\draw[color=black] (-3.3,1.9) node {\normalsize $-3$};
\draw [color=black][line width=.7pt] (5.340866869185273,1.3677329684261563)-- ++(-2.0pt,0 pt) -- ++(4.0pt,0 pt) ++(-2.0pt,-2.0pt) -- ++(0 pt,4.0pt);
\draw [color=black][line width=.7pt] (8.44631643863285,1.3677329684261563)-- ++(-2.0pt,0 pt) -- ++(4.0pt,0 pt) ++(-2.0pt,-2.0pt) -- ++(0 pt,4.0pt);
\draw [color=black][line width=.7pt] (2.2354172997376955,1.367732968426156)-- ++(-2.0pt,0 pt) -- ++(4.0pt,0 pt) ++(-2.0pt,-2.0pt) -- ++(0 pt,4.0pt);
\draw [color=black][line width=.7pt] (-0.870032269709882,1.367732968426156)-- ++(-2.0pt,0 pt) -- ++(4.0pt,0 pt) ++(-2.0pt,-2.0pt) -- ++(0 pt,4.0pt);
\draw [color=black][line width=.7pt] (11.55176600808043,1.367732968426156)-- ++(-2.0pt,0 pt) -- ++(4.0pt,0 pt) ++(-2.0pt,-2.0pt) -- ++(0 pt,4.0pt);
\draw [color=black][line width=.7pt] (-3.9754818391574593,1.367732968426156)-- ++(-2.0pt,0 pt) -- ++(4.0pt,0 pt) ++(-2.0pt,-2.0pt) -- ++(0 pt,4.0pt);
\draw [color=black][line width=.7pt] (14.657215577528014,1.367732968426156)-- ++(-2.0pt,0 pt) -- ++(4.0pt,0 pt) ++(-2.0pt,-2.0pt) -- ++(0 pt,4.0pt);
\draw[color=black] (-6.5,-0.85) node {\large $-\Lambda^+$};
\draw[color=black] (-6.35,-4.75) node {\large $C_\infty^-$};
\draw[color=black] (-6.35,-11.3) node {\large $2a^-$};
\draw [color=black] (5.340866869185273,-13.89019819677828)-- ++(-2.0pt,0 pt) -- ++(4.0pt,0 pt) ++(-2.0pt,-2.0pt) -- ++(0 pt,4.0pt);
\draw[color=black] (6.2,-13.8) node {\large $\widetilde z_0$};
\draw[color=black] (10,-7.8) node {\Large $\widetilde\Omega$};
\draw[color=black] (7,-7.8) node {\Large $\widetilde\omega$};

\draw[color=black] (7.4,-0.68) node {\normalsize $\varepsilon/2$};
\draw[color=black] (8.9,-1.65) node {\normalsize $\varepsilon/2$};
\draw[<->][>=stealth][line width=.5pt] (6.5,-0.18) -- +(0,-0.98);
\draw[<->][>=stealth][line width=.5pt] (8,-1.15) -- +(0,-0.98);

\end{scriptsize}
\end{tikzpicture}
\end{center}
\caption{Drawing of $\widetilde \Omega$, $\widetilde \omega$ and $\widetilde{z_0}$.}\label{figureomegah}
\end{figure}

We are going to prove that the number of eigenvalues of $\mathcal P$ in $\omega_h=1+2h\widetilde \omega$ is a $o(h^{1-d})$. Since there are no eigenvalues of \eqref{dampedwaveequation2} with imaginary part smaller than $2a^-$ this will prove that for every $\varepsilon>0$ the number of eigenvalues $\tau$ satisfying $\mathfrak{Re}(\tau)\in [h^{-1}-1 ; h^{-1}+1]$ and $\mathfrak{Im}(\tau) \leq -\Lambda^+-\varepsilon$ is a $o(h^{1-d})$ when $h$ tends to zero. The proof is exactly the same for eigenvalues satisfying $\mathfrak{Im}(\tau) \geq -\Lambda^- +\varepsilon$ and this will thus prove Theorem \ref{th:3}.

\paragraph*{} The key ingredient here is the next proposition but, in order to improve clarity, its proof is postponed until the next section. 

\begin{proposition}\label{prop:ResolvanteApprochee}
For every complex number $z$ in $\Omega_h=1+2h\widetilde\Omega$ there exists an operator $R(z)\in \mathcal L (L^2)$ depending holomorphically on $z\in\Omega_h$ such that $R(z)(\mathcal P-z)=\mathrm{Id} +R_1(z)+R_2(z)$ where $R_1,R_2\in \mathcal L (L^2)$, $\|R_1(z)\|_{L^2}<1/2$ and $\|R_2(z)\|_{\mathrm{tr}}= o(h^{1-d})$. Moreover for $z_0=1+2h \widetilde{z_0}\in \omega_h$ the operator $R(z_0)(\mathcal P(z_0) - z_0)$ is invertible in $L^2$ and $\left\| \left(R(z_0)(\mathcal P(z_0) - z_0)\right)^{-1}\right\|_{L^2}$ is uniformly bounded in $h$.
\end{proposition}

Using this proposition we want to bound the number of eigenvalues of $\mathcal P$ in $\Omega_h$. First of all notice that if $\mathcal P -z$ has a non zero kernel then so does $R(z)(\mathcal P -z)$, thus we only need to bound the dimension of the kernel of $R(z)(\mathcal P -z)$. Now since $\|R_1(z)\|_{L^2}\leq 1/2$ the operator $ \mathrm{Id} + R_1(z)$ is invertible and there exists an invertible operator $Q(z)$ such that $Q(z)R(z)(\mathcal P -z)= \mathrm{Id}+K(z)$ where $K(z)= Q(z)R_2(z)$. Since $Q$ is invertible we have $\mathrm{dim} \ker (R(z)(\mathcal{P}-z))=\dim \ker (\mathrm{Id}+K(z))$. 

\paragraph*{} The operator $R_2$ is of trace class and thus $K$ is also of trace class with 
\[
\|K\|_\mathrm{tr}\leq \|Q\|_{L^2}\|R_2\|_{\mathrm{tr}}\leq 2 \|R_2\|_{\mathrm{tr}}= o(h^{1-d})\pp
\]
It follows that $z$ is an eigenvalue of $\mathcal P$ only if $D(z)\overset{\mathrm{def}}{=} \det (1+K(z))$ is equal to $0$. Moreover the  multiplicity of an eigenvalue $z$ of $\mathcal P$ is less than the multiplicity of the zero of $D$. Using a general estimate on Fredholm determinants we get 
\[
|D(z)|\leq \exp(\|K(z)\|_{\mathrm{tr}})
\]
On the other hand for $z_0$ we have 
\[
|D(z_0)|^{-1}= \det((1+K(z_0))^{-1}) \text{ and } (1+K(z_0))^{-1}=1-K(z_0)(1+K(z_0))^{-1}\pp
\]
Using this and the estimate on Fredholm determinant we get
\[ |D(z_0)|^{-1}=|\det(1-K(z_0)(1+K(z_0))^{-1})|\leq \exp\left(\|K(z_0)(1+K(z_0))^{-1}\|_{\mathrm{tr}}\right)\]
\[\leq \exp\left(\|(1+K(z_0))^{-1}\| \|K(z_0)\|_{tr}\right) \]
which, according to Proposition \ref{prop:ResolvanteApprochee}, is smaller than $\exp(C \|K(z_0)\|_{\mathrm{tr}})$ for some constant $C>0$ independent of $h$. So far we have proved that
\begin{equation}\label{eq:bornedetfredholm}
\log|D(z_0)|\geq -C \|K(z_0)\|_{\mathrm{tr}}\; \text{ and }\;\forall z \in \Omega_h, \; \log|D(z)|\leq \|K(z)\|_{\mathrm{tr}}\pp
\end{equation}
As we will see in section \ref{section:preuvepropresolvanteapprochee} the operators $R(z)$ and $R_2(z)$ both depend holomorphically in $z$ on $\Omega_h$. Since $\mathcal P -z$ is holomorphic we get that $\mathrm{Id}+R_1(z)$ and $Q(z)$ are also holomorphic and finally that $K(z)$ is holomorphic in $z$ on $\Omega_h$. Consequently the function $z\mapsto D(z)$ is holomorphic on $\Omega_h$, we want to apply Jensen's inequality in order to bound the number of zeros of $D$ on a subset of $\Omega_h$ containing $z_0$.  

\paragraph*{}
\begin{figure}[ht] 
	\centering
		\includegraphics[width=1\textwidth]{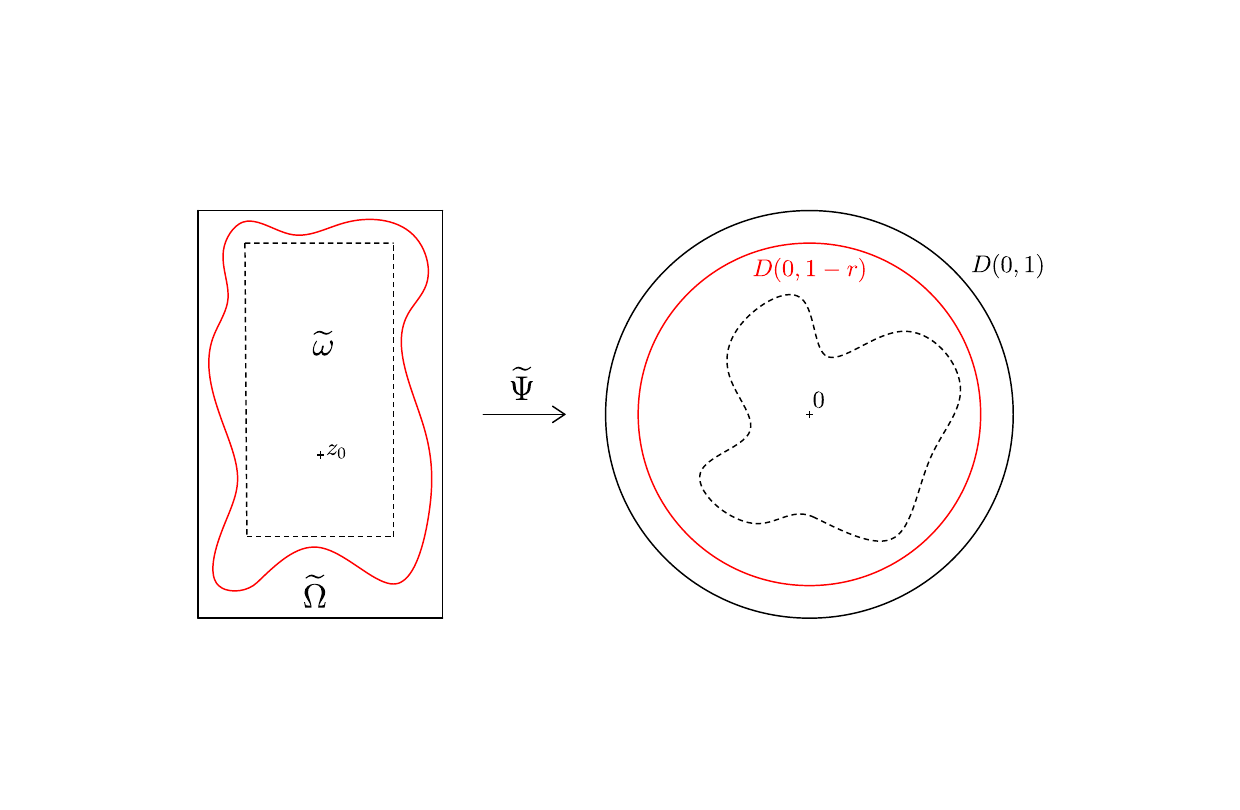}
	\caption{If $r$ is sufficiently close to $0$ then $\widetilde\Psi(\widetilde \omega)\subset D(0,1-r)$. }
	\label{fig:fonctionpsi}
\end{figure} 

 Since $\widetilde \Omega$ is a simply connected open set there exists a Riemannian mapping $\widetilde\Psi : \widetilde\Omega \to D(0;1) $ between $\widetilde\Omega$ and the open unit disk which also satisfy $\widetilde\Psi(\widetilde z_0)=0$. If we put $\Psi_h : z \mapsto \widetilde\Psi(\frac{z/2-1}{h})$ then $\Psi_h : \Omega_h\to D(0,1)$ is a Riemannian mapping which maps $z_0$ to $0$. For every $0<t<1$ let us call $n(t)$ the number of zeros (with multiplicity) of $D\circ \Psi_h^{-1}$ in $D(0,t)$; Jensen's formula states that 
\[
\frac{1}{2\pi}\int_0^{2\pi} \log|D \circ \Psi_h^{-1} (te^{i\theta}) | d\theta-\log|D\circ\Psi_h(0)|=\int_0^t \frac{n(s)}{s}d s\pp
\]
For $r$ close enough to $0$ we have $\Psi_h(\omega_h)\subset D(0;1-r)$, we then use \eqref{eq:bornedetfredholm} to obtain 
\[
\frac{1}{2\pi}\int_0^{2\pi} \log|D \circ \Psi_h^{-1} ((1-r/2)e^{i\theta}) | d\theta \leq  \sup_{z\in \Omega_h} \|K(z)\|_{\mathrm{tr}} =o(h^{1-d})\]
\[\text{ and } \log|D\circ \psi_h^{-1}(0)|\geq -C \|K(z_0)\|_{\mathrm{tr}}\pp
\]
Combining this estimates with Jensen's formula for $D\circ \Psi_h^{-1}$ gives us
\[
\int_0^{1-r/2}\frac{n(s)}{s}ds \leq (1+C) \sup_{z\in \Omega_h} \|K(z)\|_{\mathrm{tr}} = o(h^{1-d})
\]
and since the map $s\mapsto n(s)$ is increasing we have \[n(1-r)\leq \frac{2}{r}\int_0^{1-r/2}\frac{n(s)}{s}ds \leq \frac{2}{r}(1+C) \sup_{z\in \Omega_h} \|K(z)\|_{\mathrm{tr}} \]
The number of zeros of $D$ in $\omega_h$ is equal to the number of zeros of $D\circ \Psi_h^{-1}$ in $\Psi_h(\omega_h)$ which is obviously less than $n(1-r)$.

\paragraph*{} Notice that $r$ does not depend on $h$ because $\Omega_h$, $\omega_h$, $z_0$ and $\Psi_h$ are just rescaled versions of $\widetilde \Omega$, $\widetilde \omega$, $\widetilde z_0$ and $\widetilde \Psi$. We therefore obtain the desired bound : the number of zeros of $D$ in $\omega_h$ is a $o(h^{1-d})$ and the proof of Theorem \ref{th:3} is complete. 
\begin{flushright}
\qedsymbol
\end{flushright}

\section{Proof of Proposition \ref{prop:ResolvanteApprochee}}\label{section:preuvepropresolvanteapprochee}

In order to ease the notations we will use the Landau notation (or ``big O'' notation) directly for operators throughout all this section. It must always be interpreted as a Landau notation for the $L^2$ norm of the operator when the semi-classical parameter $h$ goes to zero. For example if we write $\mathcal P-z = P-1 +\mathcal O (h)$ we mean that 
\[
\left\|\mathcal P -z -P +1\right\|_{L^2}=\mathcal{O}(h)\pp
\] 

\subsection{Idea of the proof}

\paragraph*{}Notice that $\mathcal P -z =P-1+\mathcal{O}(h)$ and that the principal symbol of $P-1$ is $|\xi|_g^2-1$. Using functional calculus (see \eqref{eq:f(p-1/h)}) it is easy to find some pseudo-differential operator $A_3$ such that the principal symbol of $A_3(\mathcal P-z)$ is $1$ where $||\xi|^2_g-1|\geq Ch$ for some fixed but large enough $C$. On the set where $||\xi|^2_g-1|<Ch$ we use a different approach. We start with the formula
\[
\frac{-i}{h}\int_0^Te^{it(\mathcal{P}-z)/h}dt (\mathcal{P}-z)=1-e^{iT(\mathcal{P}-z)/h}
\]
and we would like to have $\|e^{iT(\mathcal{P}-z)/h}\|_{L^2}<1$ for some time $T$ so $1-e^{iT(\mathcal{P}-z)/h}$ would be invertible. As we will show in Lemma \ref{lemmeetatcoherents} the operator $e^{iT(\mathcal{P}-z)/h}$ behaves like 
\[e^{-iTz/h}e^{iTh\Delta}\mathrm{Op}^{\mathrm{AW}}_h(G_{2T}(x,\xi/2))\]
for functions that are concentrated around the energy layer $p^{-1}(1)$. Here $\mathrm{Op}^{\mathrm{AW}}_h$ denotes the anti-Wick quantization\footnote{See Annex \ref{an:antiwick} for a definition.} and  $\phi_t$ is the Hamiltonian flow associated with $p$. We thus see that we cannot hope to have $\|e^{iT(\mathcal{P}-z)/h}\|_{L^2}<1$ for every $z\in \Omega_h$. More precisely we cannot hope to have $\|e^{iT(\mathcal{P}-z)/h}\|_{L^2}<1$ when $\mathfrak{Im}(z/h)\geq 2C_\infty^-$. However, for almost every $(x,\xi)$ in $p^{-1}(1)$ we have $ \lim_{t\to\infty} \frac{1}{t}\log \|G_{2t}(x,\xi/2)\|\leq 2\Lambda^+ $ and so for $T$ large enough we have 
\begin{equation}\label{eq:jesaispasquoimettre}
\frac{1}{T}\log\|G_{2T}(x,\xi/2)\|\leq 2\Lambda^++\varepsilon/2
\end{equation}
for every $(x,\xi)$ of a compact set $K\subset p^{-1}(1)$ with large measure. If $\chi$ is a $C^\infty$ function with compact support in $\xi$ such that $\chi(x,\xi)=0$ for every $(x,\xi)$ that does not satisfy \eqref{eq:jesaispasquoimettre} then
\[
\| e^{iT(\mathcal{P}-z)/h}\mathrm{Op}^{\mathrm{AW}}_h(\chi)\|_{L^2}<1\pp
\]
In order to localize this to the energy layer $||\xi|_g-1|< C h$ we will use operators of the form $f\left(\frac{P-1}{h}\right)$ where $f\in \mathcal S(\mathbf R)$ is properly chosen. For instance we will show that the operator \[f\left(\frac{P-1}{h}\right)\mathrm{Op}^{\mathrm{AW}}_h(1-\chi)f\left(\frac{P-1}{h}\right)\] has a small trace norm compared to $h^{1-d}$. Then, by choosing carefully $f$, $\chi$ and $T$ depending on $h$ we can define an operator $R(z)$ which has the properties presented in Proposition \ref{prop:ResolvanteApprochee}.

\subsection{Construction of the operator $R(z)$}\label{section:constructionR(z)}
\paragraph*{} According to Oseledec's theorem for almost every $(x,\xi)\in T^*M$ the Lyapunov exponents of $(G_t(x,\xi))_{t\geq0}$ are well defined. Recall that 
\[
\Lambda^- = \underset{(x,\xi)\in T^*M}{\mathrm{ess} \inf} \lambda_1(x,\xi) \text{ and } \Lambda^+=\underset{(x,\xi)\in T^*M}{\mathrm{ess} \sup} \lambda_n(x,\xi)\vv
\]
where $\lambda_1(x,\xi)\leq \ldots \leq \lambda_n(x,\xi)$ are (if they exist) the Lyapunov exponents of $(G_t(x,\xi))_{t\geq 0}$. So for almost every $(x,\xi)$ in $\{ (x,\xi)\in T^*M : |\xi|_g =1] \}$ we have 
\[
\lim_{t\to +\infty} \frac 1 t \log (\|G_{2t}(x,\xi/2)\|)=2\lambda_n(x,\xi)\leq 2\Lambda^+\pp
\]
According to Egorov's theorem for every $\varepsilon, \eta>0$ there exist a compact set $ K\subset p^{-1}(1)$ and a time $T_0$ such that for every $T\geq T_0$ and every $(x,\xi)$ in $ K$ we have 
\[
\frac 1 T \log(\|G_{2T}(x,\xi/2)\|) < 2\Lambda^+ + \varepsilon/2
\]
and the Liouville's measure of $p^{-1}(1)\backslash  K$ is smaller than $\eta$. Now remark that the geodesic flow is continuous and so, for a fixed $T$, the matrix $G_T$ depends continuously on $(x,\xi)$. This means that if a point $(x,\xi)$ is close enough to $K$ then we still have
\[
\frac 1 T \log(\|G_{2T}(x,\xi/2)\|)< 2\Lambda^++\varepsilon/2 \pp
\]
Consequently there exists some $\mathscr C^\infty$ smooth function $\widetilde\chi_K^{(T)} : T^*M \to \mathbf R_+$ such that $\widetilde\chi_K^{(T)}$ equals $1$ on $K$, $\mathrm{supp} (\widetilde\chi_K^{(T)}) \subset \{ (x,\xi)\in T^*M : |\xi|_g \in [1/2;3/2]\}$ and 
\[
(x,\xi)\in \mathrm{supp} (\widetilde\chi_K^{(T)}) \implies \frac 1 T \log(\|G_{2T}(x,\xi/2)\|)< 2\Lambda^+ + \varepsilon/2\pp
\]
We then define the function $\chi_K^{(T)}= \widetilde \chi_K^{(T)} \circ j \circ \phi_T $ where $j(x,\xi)=(x,-\xi)$ and $\phi$ is the hamiltonian flow generated by $p$. Notice that the functions $j$ and $\phi_T$ both preserve the Liouville measure on $p^{-1}(1)$.

\paragraph*{} We now choose some non negative function $f\in \mathcal S(\mathbf R)$ such that $f(0)=1$ and $\mathrm{supp}(\hat f)$ is compact. Let $C$ be some positive constant which will be fixed later and let us define the three operators $A_1$, $A_2$ and $A_3$ as

\[
A_1(z)= -f\left(\frac{P-1}{Ch}\right)\mathrm{Op}^{\mathrm{AW}}_h(\chi_K^{(T)}) f\left(\frac{P-1}{Ch}\right)\frac i h \int_0^T e^{it(\mathcal P-z)/h} d t\vv
\]
\[
A_2=-f\left(\frac{P-1}{Ch}\right) \mathrm{Op}^{\mathrm{AW}}_h(1-\chi_K^{(T)})  f\left(\frac{P-1}{Ch}\right) \frac i h  \int_0^T e^{it(\mathcal P- z_0)/h} d t \vv
\]
\[
\text{and }\; A_3=\left(1-f\left(\frac{P-1}{Ch}\right)\right) w(P-1) \left(1-f\left(\frac{P-1}{Ch}\right)\right)\pp
\]
Here $\mathrm{Op}^{\mathrm{AW}}_h(b)$ is the $h$-anti-Wick quantization of a symbol $b$ and $w$ is defined by
\[w: x\mapsto \frac{1}{x} \left(1-\chi\left(\frac{x}{Ch}\right)\right)\]
where $ \chi$ is some smooth cut-off function with $\chi(x)=1$ around zero and some small support that will be chosen later. For a definition of the anti-Wick quantization see Annex \ref{an:antiwick}.  We finally define $R(z)$ as the sum of $A_1$, $A_2$ and $A_3$ and we are now ready to state a more precise version of Proposition \ref{prop:ResolvanteApprochee}.

\begin{proposition}\label{prop:ResolvanteApprocheeBIS}
We have the following equalities.
\begin{equation}\label{eq:a1p-z}
A_1(z)(\mathcal P-z)= f\left(\frac{P-1}{Ch}\right)\mathrm{Op}^{\mathrm{AW}}_h(\chi_K^{(T)}) f\left(\frac{P-1}{Ch}\right) +\mathcal{O}\left(\frac{1}{C}\right) +\mathcal{O}\left(e^{-T\varepsilon/2}\right) + \mathcal{O}_T(h^{1/2})
\end{equation}
\begin{equation}\label{eq:a2p-z}
A_2(\mathcal P-z_0)= f\left(\frac{P-1}{Ch}\right)\mathrm{Op}^{\mathrm{AW}}_h(1-\chi_K^{(T)}) f\left(\frac{P-1}{Ch}\right)+\mathcal{O}\left(\frac{1}{C}\right)+\mathcal{O}\left(e^{-T\varepsilon/2}\right) + \mathcal{O}_T(h^{1/2})
\end{equation}
\begin{equation}\label{eq:a3p-z}
\left\|A_3(\mathcal P-z) - \left(1-f\left(\frac{P-1}{Ch}\right)\right)^2\right\|_{L^2}\leq \mathcal{O}\left(\frac{1}{C}\right)+ \delta
\end{equation}
Where $\delta>0 $ depends only on $\chi$ and $f$ and can be made arbitrarily small. We moreover have the following bounds
\begin{equation}\label{eq:tracea2}
\left\|A_2(\mathcal P-z)\right\|_{\mathrm{Tr}}\leq \left( \int_{p^{-1}(1)}1-\chi_K^{(T)} dL_0\right) \cdot \mathcal{O}_{C}(h^{d-1})
\end{equation}
\begin{equation}\label{eq:trace1-chi}
\left\|f\left(\frac{P-1}{Ch}\right)\mathrm{Op}^{\mathrm{AW}}_h(1-\chi_K^{(T)}) f\left(\frac{P-1}{Ch}\right)\right\|_{\mathrm{Tr}}\leq \left( \int_{p^{-1}(1)}1-\chi_K^{(T)} dL_0\right) \cdot \mathcal{O}_{C}(h^{d-1})
\end{equation}
where $L_0$ is the Liouville measure on $p^{-1}(1)$.

\end{proposition}

\subsection{Proof of Proposition \ref{prop:ResolvanteApprocheeBIS}}

\paragraph*{}We start by proving a lemma on commutators with $f\left(\frac{P-1}{Ch}\right)$. 

\begin{lemme*}\label{lemmecommutateur}
Let $U$ be a pseudo-differential operator on $M$ with symbol $u$ compactly supported in the $\xi$ variable. We have the following bound.
\[
\left\|\left[f\left(\frac{P-1}{Ch}\right), U \right]\right\|_{L^2}=\mathcal{O}\left(\frac{1}{C}\right)\pp
\]
\end{lemme*}

\begin{proof}[Proof.]
We first recall the formula
\begin{equation}\label{eq:f(p-1/h)}
f\left(\frac{P-1}{Ch}\right)=\frac{1}{\sqrt{2\pi}}\int_{\mathbf R} \hat f(t) e^{it\frac{P-1}{Ch}}dt \pp
\end{equation}
According to this formula
\[
\left[f\left(\frac{P-1}{Ch}\right), U \right]=  \int_{\mathbf R} \hat f (t) \left[ e^{it\frac{P-1}{Ch}} , U \right] dt = \int_{\mathbf R} \hat f (t) e^{\frac{-it}{Ch}}\left[ e^{it\frac{P}{Ch}} , U \right] dt\pp
\]
Since $e^{it\frac{P}{Ch}}$ is an isometry we have 
\[
\left\|\left[ e^{it\frac{P}{Ch}} , U \right] \right\|_{L^2}= \left\|e^{it\frac{P}{Ch}}U e^{-it\frac{P}{Ch}}-U\right\|_{L^2}
\]
and by Egorov's theorem we know that $e^{it\frac{P}{Ch}}Ue^{-it\frac{P}{Ch}} = \mathrm{Op}_h(u\circ \phi_{t/C})+\mathcal O(h)$ where $\phi$ is the Hamiltonian flow on $T^*M$ associated with $p$. Notice that if we write $(y,\eta)=\phi_{t/C}(x,\xi)$ then $|\xi|_x^2=|\eta|_y^2$ so $u $ and $u \circ \phi_{t/C}$ are both compactly supported in $\xi$. Consequently we have 
\[
\left\|u-u \circ \phi_{t/C}  \right\|_\infty = \mathcal O\left(\frac{1}{C}\right)
\]
uniformly in $t\in \mathrm{supp}(\hat f)$ and the same estimates goes with all the derivatives of $u - u \circ \phi_{t/C}$. According to Calderon-Vaillancourt's theorem we then have 
\[
\left\|\left[ e^{it\frac{P}{Ch}} , U\right] \right\|_{L^2}=\mathcal O \left(\frac{1}{C}\right)
\]
uniformly in $t\in\mathrm{supp }(\hat f)$ and we finally get
\[
\left\|\left[f\left(\frac{P-1}{Ch}\right), U \right]\right\|_{L^2} = \mathcal O \left(\frac{1}{C}\right)\pp
\]
\end{proof}

We also need a lemma to approximate $e^{it(\mathcal{P}-z)/h}$.

\begin{lemme*}\label{lemmeetatcoherents}
Let $t\in [0;T]$ be a fixed positive real number and $u$ be a  symbol compactly supported, we have
\[
\left\|\mathrm{Op}^{\mathrm{AW}}_h(u)\left[e^{it\mathcal P/h}-\mathrm{Op}^{\mathrm{AW}}_h(G_{2t}\circ j \circ \phi_t(x,\xi/2)) e^{-ith\Delta} \right] \right\|_{L^2(M)}=\mathcal O_T(h^{1/2})
\]
where $j(x,\xi)=(x,-\xi)$ for every $(x,\xi)\in T^*M$ and $\phi_t$ is the Hamiltonian flow generated by $p$.

\end{lemme*}

In other words the damped propagator can be, at the lowest order, factorized by the undamped propagator and a damping part which principal symbol is given by the function $G$.  We here give the same proof as in \cite{thesis} but an arguably more digest proof of a similar result can be found for example in Lemma 6 of  \cite{non11}.
 
\begin{proof}[Proof.]
We are only going to prove this result with $M=\mathbf R^d$ with a metric $g$, the extension  to any compact Riemannian manifold is straightforward.  The first step of the proof is to precisely describe the action of $e^{it\mathcal P/h}$ on coherent state, this is a classical result and we will follow the presentation and notations of \cite{rob06}. Let $g : \mathbf R^d \to \mathbf C$ be the function defined by
\[
g : x\mapsto \frac{1}{\pi^{d/4}}\exp\left(-\|x\|_2^2/2\right)\pp
\]
and we note $\varphi_0=\Lambda_h g$ where $\Lambda_h$ is the dilatation operator defined by $\Lambda_h f(x)=h^{-d/4}f(h^{-1/2}x)$. In other words we have
\[
\varphi_0(x)=\frac{1}{(\pi h)^{d/4}}\exp\left(\frac{-\|x\|^2}{2h}\right)\pp
\]
If $\rho=(x_0,\xi_0)$ is a point of $T^*\mathbf R^d=\mathbf R^{2d}$ we define $\varphi_\rho$ by
\[
\varphi_\rho=\mathcal T (\rho) \varphi_0
\]
where $\mathcal T(\rho)=\exp\left(\frac{i}{h}(\xi_0\cdot x +ihx_0\cdot\partial_x )\right)$ is the Weyl operator, in other words 
\[
\varphi_\rho(x)=e^{ix\cdot\xi_0/h}\varphi_0(x-x_0)\pp
\]
The function $\varphi_\rho$ is called the coherent state associated with $(x_0,\xi_0)$. Finally, if $v\in \mathbf C^n$ we define $\varphi_{\rho,v}=\varphi_\rho \cdot v$.
According to \cite{rob06}, for every integer $N$, every $\rho$ in a compact set $K$ and every $t\in [0;T]$ we have
\[
\left\|e^{it\mathcal P/h}\varphi_{\rho,v}-\psi^{(N)}_{\rho,v}(t)\right\|_{L^2}=\mathcal O_{K,T}(h^{\frac{N+1}{2}})
\]
where 
\[
\psi^{(N)}_{\rho,v}(t)=e^{i\delta_t/h}\mathcal T(\rho_t)\Lambda_h \widetilde G_{2t}( j(\widetilde\rho)) \mathcal M[F_t]\left( \sum_{0\leq j \leq N} h^{j/2}b_j(t) g \right)\pp
\]
We will only give a partial description of the terms of $\psi^{(N)}_{\rho,v}(t)$ here, for a complete definition see \cite{rob06}. The point $\rho_t\in \mathbf T^*\mathbf R^d$ is simply given by the inverse Hamiltonian flow : $\rho_t=(x_{-t},\xi_{-t})=\phi_{-t}(x_0,\xi_0)=\phi_{-t} \rho$ and $\widetilde \rho =(x_0,\xi_0/2)$. The quantity $\delta_t$ is real and only depends on $t$
. The function $\widetilde G_t : T^*\mathbf R^d \to \mathscr M_n(\mathbf C) $ is defined as the solution of the following differential equation :
\begin{equation}
\left\lbrace
\begin{array}{l}
\widetilde G_0(x_0,\xi_0)=\mathrm{Id}_n \\
\partial_t \widetilde G_t(x_0,\xi_0)=-a(x_{t})\sqrt z \widetilde G_{t}(x_0,\xi_0)
\end{array}
\right.
\end{equation}
we thus have $\dt \widetilde G_{2t}(j(\widetilde \rho)) = -2a(x_{-t})\sqrt z \widetilde G_{2t}(j(\widetilde \rho))$. The function $j: T^*{\mathbf R^d}\to T^*\mathbf R^d$ is defined by $j(x,\xi)=(x,-\xi)$. The functions $b_j(t) $ are polynomial functions in the $x$ variable with coefficients in $\mathbf C^n$, the first term $b_0(t)$ is constant and equal to $v$. The term $\mathcal M[F_t]$ describe the sprawl of the Gaussian $g$ under the action of the propagator, $F_t$ is a flow of linear symplectic transformation and $\mathcal M$ is a realization of the metaplectic representation. We can apply the same result when $a=0$ and we find 
\[
\left\|e^{-iht \Delta}\varphi_{\rho,v}-e^{i\delta_t/h}\mathcal T(\rho_t)\Lambda_h \mathcal M[F_t]\left( \sum_{0\leq j \leq N} h^{j/2}c_j(t) g \right)\right\|_{L^2}=\mathcal O_{K,T}(h^{\frac{N+1}{2}})\vv
\]
the only difference with $e^{it\mathcal P/h}\varphi_{\rho,v}$ is in the polynomials $c_j$ and in the absence of $\widetilde G_t$, note that we also have $c_0(t)=v$. Since the function $v\mapsto \psi_{\rho,v}^{(N)}(t)$ is linear and since $b_0(t)=c_0(t)=v$ we can define by induction some matrices $q_1(t,\rho),\ldots , q_N(t,\rho)$ depending polynomially on $x$ such that 
\[
\left(1+h^{1/2}q_1(t,\rho)+\ldots + h^{N/2}q_N(t,\rho)\right)\left( \sum_{0\leq j \leq N} h^{j/2}c_j(t) g \right) =
\]
\begin{equation}\label{eq:definitiondesqi}
 \left( \sum_{0\leq j \leq N} h^{j/2}b_j(t) g \right) + \mathcal O(h^{\frac{N+1}{2}})
\end{equation}
for every $v\in \mathbf C^n$, $t\in [0;T]$ and $\rho\in K$. Indeed the matrix $q_i(t,\rho)$ must satisfy the relation 
\[
c_i(t)+\sum_{j=1}^{i} q_j(t,\rho)c_{i-j}(t)=b_i(t)\Longleftrightarrow q_i(t,\rho)v=b_i(t)-c_i(t)-\sum_{j=1}^{i-1} q_j(t,\rho)c_{i-j}(t)
\]
for every $v\in \mathbf C^n$ and since all the polynomials $b_j(t)$ and $c_j(t)$ depend linearly on $v$ this defines uniquely the matrix $q_i(t,\rho)$ if the $q_j(t,\rho), \; j=1,\ldots,i-1$ are fixed.
Now let $f_1,f_2 : \mathbf R^{2d} \to \mathscr M_n(\mathbf C)$ be two symbols of order $\leq m\in \mathbf R$; then we have
\[
\mathcal{M}[F_t]\mathrm{Op}_1^{\mathrm W}(f_1)=\mathrm{Op}_1^{\mathrm W}\left(f_1\circ F_t^{-1}\right) \mathcal M[F_t]
\]
where $\mathrm{Op}_1^{\mathrm W}$ is the Weyl quantization for $h=1$. A proof of this result can be found in \cite{rob06}. We also have
\[
\Lambda_h(f_1\cdot f_2)=h^{d/4}\Lambda_h(f_1)\cdot\Lambda_h(f_2)\vv
\]
\[
\text{ and }\; \mathcal{T}(z_t)(f_1 \cdot f_2) = \mathcal{T}(z_t)(f_1)\cdot\mathcal{T}(z_t)(f_2)\pp
\]
Using the previous relations and \eqref{eq:definitiondesqi} we see that there exist some matrices $Q_1(t,\rho),\ldots,Q_N(t,\rho)$ depending polynomially on $x$ such that 
\[
\left\|\psi_{\rho,v}^{(N)} -\left(\widetilde G_{2t}(j(\widetilde\rho))+\sum_{i=1}^N h^{i/2}Q_i(t,\rho)\right)e^{i\delta_t/h}\mathcal T(\rho_t)\Lambda_h \mathcal M[F_t]\left( \sum_{0\leq j \leq N} h^{j/2}c_j(t) g \right)\right\|_{L^2}=\mathcal O(h^{\frac{N+1}{2}})\pp
\]
Notice that the matrix $\widetilde G_t(j(\rho))$ does not depend on $x$, this is due to the fact that $b_0(t)=c_0(t)$. Consequently we have 
\begin{equation}
\left\| e^{it\mathcal P/h}\varphi_{\rho,v}-\left(\widetilde G_{2t}(j(\widetilde\rho))+\sum_{i=1}^N h^{i/2} Q_i(t,\rho)\right)e^{-iht\Delta}\varphi_{\rho,v} \right\|_{L^2} = \mathcal O(h^{\frac{N+1}{2}})
\end{equation}
and the $\mathcal O(h^{\frac{N+1}{2}}) $ is uniform in $\rho\in  K$ and $t\in [0;T]$. If we take some symbol $u$ with $\supp (u) \Subset K$ then the estimates become uniform in $\rho $ : 
\begin{equation}
\left\| \mathrm{Op}_h^{\mathrm{AW}}(u) e^{it\mathcal P/h}\varphi_{\rho,v}- \mathrm{Op}_h^{\mathrm{AW}}(u)\left(\widetilde G_{2t}(j(\widetilde\rho))+\sum_{i=1}^N h^{i/2}Q_i(t,\rho)\right)e^{-iht\Delta}\varphi_{\rho,v} \right\|_{L^2} = \mathcal O(h^{\frac{N+1}{2}})
\end{equation}
uniformly in $t\in [0;T]$ and $\rho\in \mathbf R^{2d}$. The next step is to remark that $e^{-iht\Delta}\varphi_{\rho,v}$ is a sum of derivatives of  anisotropic coherent states associated with $\rho_t$. So, by  using $\rho=\phi_{t} \rho_t$, there exist some symbols $g_t^{(i)}$ such that
\[
\left\|\mathrm{Op}_h^{\mathrm{AW}}(u)\left[\widetilde G_{2t}(j(\widetilde\rho))e^{-iht\Delta}-\left(\mathrm{Op}_h^{\mathrm{AW}}(\widetilde G_{2t}\circ j \circ \phi_{2t})+\right. \right. \right. \quad \quad \quad \quad \quad \quad  \quad \quad \quad \quad  
\]
\[
\quad \quad  \quad \quad  \quad \quad \quad \quad  \quad  \quad \left.\left.\left.\sum_{i=1}^{N}h^{j/2}\mathrm{Op}_h^{\mathrm{AW}}(g_t^{(i)})\right)e^{-iht\Delta}\right]\varphi_{\rho,v}\right\|_{L^2}=\mathcal O(h^{\frac{N+1}{2}})
\]
uniformly in $\rho\in \mathbf R^{2d}$ and $t\in[0;T]$. The same  is true for the matrices $Q_i$ and thus there exist some symbols $G_t^{(i)}$ such that
\begin{equation}\label{eq:evolutiondevarphizv}
\left\|\mathrm{Op}_h^{\mathrm{AW}}(u)\left[  e^{it\mathcal P/h}- \left(\mathrm{Op}_h^{\mathrm{AW}}(\widetilde G_{2t}\circ j \circ \phi_{2t})+\sum_{i=1}^{N}h^{j/2}\mathrm{Op}_h^{\mathrm{AW}}(G_{2t}^{(i)})\right)e^{-iht\Delta} \right]\varphi_{\rho,v}\right\|_{L^2}=\mathcal O(h^{\frac{N+1}{2}})\pp
\end{equation}
We now use the fact that for every function $f=(f_1,\ldots, f_n)\in  L^{2}(\mathbf R^d)^n$ we have 
\begin{equation}\label{eq:identitequantificationantiwick}
f_i =\frac{1}{(2\pi h)^d}\int_{\mathbf R^{2d}} \langle f_i , \varphi_{z,e_i}\rangle_{L^2(\mathbf R^d)} f_i dz 
\end{equation}
where $e_i=(0,\ldots , 0 ,1, 0 \ldots, 0)$ is the $i$-th vector of the canonical basis of $\mathbf C^n$. Combining \eqref{eq:evolutiondevarphizv} and \eqref{eq:identitequantificationantiwick} we get that 
\[
\left\|\mathrm{Op}_h^{\mathrm{AW}}(u)\left[ e^{it\mathcal P/h}- \left(\mathrm{Op}_h^{\mathrm{AW}}(\widetilde G_{2t}\circ j \circ \phi_{2t})+\sum_{i=1}^{N}h^{j/2}\mathrm{Op}_h^{\mathrm{AW}}(G_{2t}^{(i)})\right)e^{-iht\Delta}  \right]\right\|_{L^2}=\mathcal O(h^{\frac{N+1}{2}-d})\pp
\]
Notice that $\|\mathrm{Op}_h^{\mathrm{AW}}(G_{2t}^{(i)})\|_{L^2}$ is bounded uniformly in $h$ and so, if we take $N+1>2d$ and only keep the principal terms we get
\[
\left\|\mathrm{Op}_h^{\mathrm{AW}}(u) \left( e^{it\mathcal P/h}- \mathrm{Op}_h^{\mathrm{AW}}(\widetilde G_{2t} \circ j\circ \phi_{2t}(x,\xi/2))e^{-iht\Delta} \right)\right\|_{L^2}=\mathcal O_T (h^{1/2})\pp
\]
The last final step is to remark that $\widetilde G_t$  depends smoothly on $z=1+\mathcal O(h)$ :  we have $\widetilde G_t = G_t$ when $z=1$ and so $\widetilde G_t = G_t +\mathcal O_T(h)$. Plugging this in the previous equality, we get what we wanted :
\[
\left\|\mathrm{Op}_h^{\mathrm{AW}}(u) \left( e^{it\mathcal P/h}- \mathrm{Op}_h^{\mathrm{AW}}(G_t \circ j\circ \phi_t)e^{-iht\Delta} \right)\right\|_{L^2}=\mathcal O_T (h^{1/2})\pp
\]
This finishes the proof of the lemma.
\end{proof}
We can now prove \eqref{eq:a1p-z}; we start by writing 
\[
\frac{-i}{h}\int_0^T e^{it(\mathcal P-z)/h} dt (\mathcal{P}-z)=\mathrm{Id}-e^{iT(\mathcal{P}-z)/h}
\]
and so
\[
A_1(z)(\mathcal{P}-z)=f\left(\frac{P-1}{Ch}\right)\mathrm{Op}^{\mathrm{AW}}_h(\chi_K^{(T)})f\left(\frac{P-1}{Ch}\right)\left( \mathrm{Id}-e^{iT(\mathcal{P}-z)/h}\right)\pp
\]
Using the fact that $\mathrm{Op}^{\mathrm{AW}}_h(\chi_K^{(T)})=\mathrm{Op}_h(\chi_K^{(T)})+\mathcal{O}_T(h)$ twice and applying Lemma \ref{lemmecommutateur} we get
\[
\displaystyle f\left(\frac{P-1}{Ch}\right)\mathrm{Op}^{\mathrm{AW}}_h(\chi_K^{(T)})f\left(\frac{P-1}{Ch}\right)e^{-iT(\mathcal{P}-z)/h}=\]
\[
 f\left(\frac{P-1}{Ch}\right)^2\mathrm{Op}^{\mathrm{AW}}_h(\chi_K^{(T)})e^{-iT(\mathcal{P}-z)/h} +\mathcal{O}_T(h)+\mathcal{O}\left(\frac{1}{C}\right)\pp
\]
The operator norm of $f\left(\frac{P-1}{Ch}\right)^2$ is bounded by $\|f^2\|_\infty$ and it only remains to estimate the operator norm of $\mathrm{Op}^{\mathrm{AW}}_h(\chi_K)e^{iT(\mathcal{P}-z)/h}$. According to Lemma \ref{lemmeetatcoherents} we have
\[\begin{array}{rcl}
\displaystyle \mathrm{Op}^{\mathrm{AW}}_h(\chi_K^{(T)})e^{iT(\mathcal{P}-z)/h}&=&\displaystyle e^{-itz/h}\mathrm{Op}^{\mathrm{AW}}_h(\chi_K^{(T)})\mathrm{Op}^{\mathrm{AW}}_h(G_{2T}\circ j \circ \phi_T(x,\xi/2))e^{-ihT\Delta}\\ \,&\, &+\displaystyle \mathcal{O}_T(h^{1/2})\\
&=& \displaystyle e^{-itz/h}\mathrm{Op}^{\mathrm{AW}}_h(\chi_K^{(T)} G_{2T}\circ j\circ \phi_T(x,\xi/2))e^{-ihT\Delta}+\mathcal{O}_T(h^{1/2})
\end{array}
\]
Since $e^{-ihT\Delta}$ is an isometry we have 
\[
\left\|\mathrm{Op}^{\mathrm{AW}}_h(\chi_K^{(T)})e^{iT(\mathcal P-z)/h}\right\|_{L^2}\leq \|\chi_K^{(T)} G_{2T}\circ j \circ \phi_T(x,\xi/2)\|_{\infty}e^{T\mathfrak{Im}(z/h)} + \mathcal{O}_T(h^{1/2})\vv
\]
recall that $\chi_K^{(T)}=\widetilde \chi_K^{(T)}\circ j \circ \phi_T$ so 
\[
\|\chi_K^{(T)} G_{2T}\circ j \circ \phi_T(x,\xi/2)\|_{\infty}=\|\widetilde\chi_K^{(T)} G_{2T}(x,\xi/2)\|_{\infty}\pp
\]
If we then use the definition of $\widetilde\chi_K^{(T)}$, $G_T$ and $\Omega_h$ we see that
\[
\|\widetilde\chi_K^{(T)} G_{2T}(x,\xi/2)\|_{\infty}e^{-T\mathfrak{Im}(z/h)}\leq e^{-T\varepsilon/2}\vv
\]
which finishes the proof of \eqref{eq:a1p-z}. The same technique is used to prove \eqref{eq:a2p-z} except that we don't even have to use Lemma \ref{lemmeetatcoherents} because $\mathfrak{Im}(z_0/h)$ is small enough :

\[
\left\|e^{iT(\mathcal{P}-z_0)}\right\|_{L^2}\leq e^{-T}\pp
\]

\paragraph*{}We continue by proving \eqref{eq:a3p-z}. Recall that $\mathcal{P}-z=P-1+\mathcal O (h)$ and that $\|w\|_{\infty}=\mathcal{O}\left(\frac{1}{Ch}\right)$ so
\[
\begin{array}{rcl}
\displaystyle A_3(z)(\mathcal{P}-z)&=&\displaystyle A_3(z)(P-1)+\mathcal{O}\left(\frac{1}{C}\right)\\
\,&=& \displaystyle \left(1-f\left(\frac{P-1}{Ch}\right)\right)w(P-1)(P-1)\left(1-f\left(\frac{P-1}{Ch}\right)\right)+\mathcal{O}\left(\frac{1}{C}\right)\pp
\end{array}\]
According to the definition of $w$ whe have $w(P-1)(P-1)= 1-\chi\left(\frac{P-1}{Ch}\right)$ and so it only remains to estimate the operator norm of 
\[
\left(1-f\left(\frac{P-1}{Ch}\right)\right)\chi\left(\frac{P-1}{Ch}\right)\left(1-f\left(\frac{P-1}{Ch}\right)\right) \pp
\]
The norm of this operator is bounded by $\|(1-f)\chi(1-f)\|_\infty$, recall that $f(0)=1$ and that $f$ is continuous so by choosing $\chi$ with sufficiently small support around $0$ we get 
\[
\|(1-f)\chi(1-f)\|_\infty <\delta
\]
for any fixed positive $\delta$ and \eqref{eq:a3p-z} is proved.

\paragraph*{} It only remains now to prove \eqref{eq:tracea2} and \eqref{eq:trace1-chi}, we start with \eqref{eq:trace1-chi} and we proceed as in \cite{sjo00}. Since the Anti-Wick quantification is positive and since $1-\chi_K^{(T)}\geq 0$ we know that
\[
\left\|f\left(\frac{P-1}{Ch}\right)\mathrm{Op}_{h}^{\mathrm{AW}}(1-\chi_K^{(T)})f\left(\frac{P-1}{Ch}\right)\right\|_{\mathrm{Tr}}\]
\[
\begin{array}{rcl}
&=&\displaystyle\mathrm{Tr}\left[ f\left(\frac{P-1}{Ch}\right)\mathrm{Op}_{h}^{\mathrm{AW}}(1-\chi_K^{(T)})f\left(\frac{P-1}{Ch}\right) \right]\\

&=&\displaystyle\mathrm{Tr}\left[ f\left(\frac{P-1}{Ch}\right)^2\mathrm{Op}_{h}^{\mathrm{AW}}(1-\chi_K^{(T)}) \right]\\
&=& \displaystyle \frac{1}{\sqrt{2\pi}} \mathrm{Tr}\left[\int_{\mathbf R} \widehat{f^2}(t)e^{it\frac{P-1}{Ch}}\mathrm{Op}_{h}^{\mathrm{AW}}(1-\chi_K^{(T)}) dt \right] \\
&=& \displaystyle \frac{C}{\sqrt{2\pi}} \mathrm{Tr}\left[\int_{\mathbf R} \widehat{f^2}(Ct)e^{it\frac{P-1}{h}}\mathrm{Op}_{h}^{\mathrm{AW}}(1-\chi_K^{(T)}) dt \right]\pp
\end{array}
\]
For $C$ large enough we have 
\begin{equation}\label{eq:supportfouriertransformoff}
\mathrm{supp}\widehat{ f^2}(C\cdot) \subset ]-\frac 1 2 T_\mathrm{min} ;\frac 1 2 T_\mathrm{min} [ 
\end{equation} 
where $T_\mathrm{min}$ is the smallest possible length of a closed trajectory in $p^{-1}(1)$ for the Hamiltonian flow generated by $p$. Whenever \eqref{eq:supportfouriertransformoff} is satisfied we know that 
\begin{equation}
\lim_{h\to 0} \mathrm{Tr}\left[\int_{\mathbf R} \widehat{f^2}(Ct)e^{it\frac{P-1}{h}}\mathrm{Op}_{h}^{\mathrm{AW}}(1-\chi_K^{(T)}) dt \right] = C_dh^{1-d}\widehat{f^2}(0)\int_{p^{-1}(1)}1-\chi_K^{(T)} L_0(d\rho)
\end{equation}
where $L_0$ is the Liouville measure on $p^{-1}(1)$ and $C_d$ only depends on $d$, the dimension of $M$. One can find a proof of this classical fact in \cite{disj99} for instance. We now use the formula $\|AB\|_\mathrm{Tr}\leq\|A\|\|B\|_\mathrm{Tr}$ and the fact that $\left\|f\left(\frac{P-1}{Ch}\right)\right\|_{L^2}\leq \left\|f\right\|_\infty$ to get \eqref{eq:trace1-chi}. We use the same technique for \eqref{eq:tracea2} and so it only remains to show that 
\[
-\frac i h  \int_0^T e^{it(\mathcal P- z_0)/h} d t (\mathcal{P}-z)
\]
is uniformly bounded in $h$. We recall that $\mathcal P -z_0$ is invertible and so we can write
\[
-\frac i h  \int_0^T e^{it(\mathcal P- z_0)/h} d t (\mathcal{P}-z)= \left(\mathrm Id - e^{iT(\mathcal{P}-z_0)/h}\right)(\mathcal P-z_{0})^{-1}(\mathcal P-z)\pp
\]
By using the energy formula \eqref{eq:formuleenergie''} we see that $\left\|e^{iT(\mathcal{P}-z_0)}\right\|_{L^2}\leq e^{-T}\leq  1$ and that $\left\|(\mathcal P -z_0)^{-1}\right\|_{L^2}$ is uniformly bounded in $h$. Consequently the operator $(\mathcal P-z_{0})^{-1}(\mathcal P-z)$ is uniformly bounded in $L^2$ norm when $h$ goes to $0$. This finishes the proof of Proposition \ref{prop:ResolvanteApprocheeBIS}.

\subsection{End of the proof of Proposition \ref{prop:ResolvanteApprochee}}
We now use Proposition \ref{prop:ResolvanteApprocheeBIS} to prove Proposition \ref{prop:ResolvanteApprochee}. According to Proposition \ref{prop:ResolvanteApprocheeBIS} we have
\[\begin{array}{rcl}
\displaystyle R(z)(\mathcal P-z )&=&\displaystyle\mathrm{Id}-f\left(\frac{P-1}{h}\right)  \mathrm{Op}_{h}^{\mathrm{AW}}(1-\chi_K^{(T)})  f\left(\frac{P-1}{h}\right) +A_2(\mathcal P-z) \\
& & \displaystyle +\mathcal O\left(\frac{1}{C}\right) +\mathcal{O}(\delta)+\mathcal{O}(e^{-\varepsilon T/2}) +\mathcal{O}_T(h^{1/2})\pp
\end{array}\]
We define
\[
R_2(z)=-f\left(\frac{P-1}{h}\right)  \mathrm{Op}_{h}^{\mathrm{AW}}(1-\chi_K^{(T)})  f\left(\frac{P-1}{h}\right) +A_2(\mathcal P-z)
\]
\[
\text{and }\;R_1(z)=R(z)(\mathcal{P}-z)-R_2(z)\pp
\]
We start by fixing a constant $C$ large enough so the remainder $\mathcal O \left(\frac{1}{C}\right)$ is smaller than $1/100$. We then chose $\chi$ in the definition of $A_3$ so that the remainder $\mathcal O (\delta)$ is smaller than $1/100$. Fix some arbitrary $\eta>0$, for $T$ large enough the remainder $\mathcal O (e^{-\varepsilon T/2})$ is smaller than $1/100$ and there exists some $\chi_K^{(T)}$ such that 
\[
\limsup_{h\to 0} h^{n-1}\|R_2(z)\|_{\mathrm Tr} <\eta\pp
\]
We finally take $h$ small enough so that the remainder $\mathcal{O}_T(h^{1/2}) $ is also smaller than $1/100$. By doing so we have constructed an operator $R(z)$ such that $R(z)(\mathcal P-z)=\mathrm{Id}+R_1(z)+R_2(z)$ with $\|R_1(z)\|_{L^2}<1/10$ and $\left\|R_2(z)\right\|_\mathrm{Tr}\leq \eta h^{1-d}$ for $h$ small enough. When $h$ goes to $0$ we can repeat the same process and make $\eta$ arbitrarily small and get $\left\|R_2(z)\right\|_\mathrm{Tr}= o (h^{1-d})$. Moreover according to Proposition \ref{prop:ResolvanteApprocheeBIS} we have 
\[
\|R_2(z_0)\|_{L^2}=\mathcal O \left( \frac{1}{C}\right)+\mathcal{O}(e^{-T\varepsilon/2})+\mathcal{O}_T(h^{1/2}) 
\]
and as before we can choose $C$ and $T$ in order to also have $\|R_2(z_0)\|_{L^2}<1/10$ for $h$ small enough. This implies that $R(z_0)(\mathcal{P}-z_0)$ is invertible and that $\left\| \left(R(z_0)(\mathcal P - z_0)\right)^{-1}\right\|$ is uniformly bounded in $h$. The proof of Proposition \ref{prop:ResolvanteApprochee} is thus finished.

\clearpage{\pagestyle{empty}\cleardoublepage}
\appendix
\section{Semi-classical anti-Wick quantization}\label{an:antiwick}

 \paragraph*{}In this appendix we present, without any proof, a construction of the $h$-Anti-Wick quantization and some of its basic properties. The proofs for $h=1$ can be found in \cite{gos11}.  We start by constructing it on $\mathbf R^d$ for scalar valued symbols. 

\begin{definition}
Let $(x,\xi)$ be a point of $T^*\mathbf R^d=\mathbf R^{2d}$, we define the function $e_{(x,\xi)}: \mathbf R^d \to \mathbf R^d$ by
\[
e_{(x,\xi)}: y\mapsto \frac{1}{(h\pi )^{d/4}}e^{-\|x-y\|_2^2/2h}e^{iy\cdot \xi /h}\pp
\]
\end{definition}
A standard calculation shows that $\|e_{(x,\xi)}\|_{L^2(\mathbf R^d)}=1$. We now define $\Pi_{(x,\xi)} :L^2(\mathbf R^d)\to L^2(\mathbf R^d)$ as the orthogonal projector on the vector subspace generated by $e_{(x,\xi)}$.

\begin{definition}
Let $a\in S^0_{1,0}(\mathbf R^{2d})$, we define its $h$-anti-Wick quantization $\mathrm{ Op}^{\mathrm{AW}}_h(a):L^2(\mathbf R^d)\to L^2(\mathbf R^d)$ by
\[
\mathrm{ Op}^{\mathrm{AW}}_h(a)=\frac{1}{(2\pi h)^d}\int_{\mathbf R^{2d} } a(x,\xi)\Pi_{(x,\xi)} dxd\xi\pp
\]
\end{definition}
The anti-Wick quantization has a few convenient properties.
\begin{proposition}\label{prop:normeantiwick}
If $a$ is real valued and non negative then $\mathrm{ Op}^{\mathrm{AW}}_h(a)$ is self-adjoint and positive:
\[
\forall u\in L^2(\mathbf R^d), \; \left\langle\mathrm{ Op}^{\mathrm{AW}}_h(a) u , u\right\rangle_{L^2(\mathbf R^d)}\geq 0 \pp
\]
Moreover we have the following estimates
\[
\left\|\mathrm{ Op}^{\mathrm{AW}}_h(a)\right\|_{L^2(\mathbf R^d)}\leq \|a\|_\infty 
\]
and 
\[
\mathrm{ Op}^{\mathrm{AW}}_h(1)=\mathrm{Id}_{L^2(\mathbf R^d)}\pp
\]
\end{proposition}
The $h$-anti-Wick quantization is linked to the $h$-Weyl quantization in the following way.
\begin{proposition}
$\mathrm{ Op}^{\mathrm{AW}}_h(a)$ is a $h$-pseudo differential operator of order $\leq 0$ and we have
\[
\mathrm{ Op}^{\mathrm{AW}}_h(a)=\mathrm{ Op}^{\mathrm{W}}_h(a*\varepsilon)
\]
where $\varepsilon : (y,\eta) \mapsto  (h\pi)^{-n/2}e^{-\|(y,\eta)\|_2^2/h}$. 
\end{proposition}

Consequently we know that $\left\|\mathrm{ Op}^{\mathrm{AW}}_h(a)-\mathrm{ Op}^{\mathrm{W}}_h(a)\right\|_{L^{2}(\mathbf R^d)}=\mathcal O(h)$ when $h$ goes to $0$. The same construction can be used for symbols $a$ valued in $\mathscr M_n(\mathbf C)$, the operator $\mathrm{ Op}^{\mathrm{AW}}_h(a)$ then acts on $L^2(\mathbf R^d)^n$. The previous results still hold in this case but Proposition \ref{prop:normeantiwick} needs to be slightly modified:

\begin{proposition}
If $a$ is valued in $\mathscr H^+_n(\mathbf C)$ the space of Hermitian positive semi-definite matrices then $\mathrm{ Op}^{\mathrm{AW}}_h(a)$ is self-adjoint and positive:
\[
\forall u\in L^2(\mathbf R^d)^n, \; \left\langle\mathrm{ Op}^{\mathrm{AW}}_h(a) u , u\right\rangle_{L^2(\mathbf R^d)^n}\geq 0 \pp
\]
Moreover we have the following estimates
\[
\left\|\mathrm{ Op}^{\mathrm{AW}}_h(a)\right\|_{L^2(\mathbf R^d)^n}\leq \sup_{(x,\xi)\in\mathbf R^d} \|a(x,\xi)\|_2 
\]
and 
\[
\mathrm{ Op}^{\mathrm{AW}}_h(\mathrm{Id}_n)=\mathrm{Id}_{L^2(\mathbf R^d)^n}\pp
\]
\end{proposition} 
We can then define a $h$-anti-Wick quantization on a manifold $M$ using a partition of unity.


\section{Multiplicative ergodic theorem of Oseledets}\label{an:oseledets}
In this appendix we present, without any proofs, the multiplicative ergodic theorem of Oseledets and some related results. The proofs can be found in \cite{led82} and \cite{bape13}.  Let $(X,\mu)$ be a probability space and let $(\varphi_t)_{t\in \mathbf R}$ be a one parameter group of measure preserving functions from $X$ to $X$. Let $G : \mathbf R \times X \to \mathscr M_n(\mathbf C)$ be a cocycle, this means that $G$ satisfies the following conditions :
\begin{itemize}
\item[-] $\forall x\in X$, $G(0,x)=\mathrm{Id}_n$,
\item[-] $\forall x\in X$, $\forall s,t \in \mathbf R$, $G(s+t,x)=G(s,\varphi_t(x))G(t,x)$.
\end{itemize} 

\begin{theo*}[Oseledets]\label{th:oseledets}
Assume that $\log\|G(t,\cdot)\|$ and $\log\|G(t,\cdot)^{-1}\|$ are both in $L^1(X,\mu)$ for every $t\in [0;1]$. For $\mu$-almost every $x\in X$ there exists real numbers $\lambda_1(x)< \ldots < \lambda_{k(x)}(x)$ and a decomposition $\mathbf C^n=V_1^x\oplus \ldots \oplus V_{k(x)}^x$ such that for every $v\in V_i^x\backslash \{0\}$  
\[
\lim_{t\to \pm\infty} \frac{1}{t} \log \left\|G(t,x)v\right\|=\lambda_i(x)\pp
\]
Moreover, the $\lambda_i$ are invariant by $\varphi_t$ : $\lambda_i(x)=\lambda_i(\varphi(x))$ and 
\[
G(t,x)V_i^x=V_i^{\varphi_t(x)}\pp
\]
\end{theo*}

The numbers $\lambda_i$ are called the Lyapunov exponents of $G$ and $\dim V_i^x$ is called the multiplicity of the Lyapunov exponent $\lambda_i(x)$. If the dynamical system $(X,\mu,(\varphi_t)_{t\in \mathbf R})$ is ergodic then the Lyapunov exponents and their multiplicity are constant on a full measure set of $X$. Note that the choice of the norm over the space $\mathscr M_n(\mathbf C)$ does not matter since they are all equivalent. Let $x\in X$ be a point for which the Lyapunov exponents are well defined, and let $\mu_1\leq \ldots \leq \mu_n(x)$ be the Lyapunov exponents counted with multiplicity, then 
\[
\sum_{j=0}^{i-1}\mu_{n-j}(x)=\lim_{t\to \pm \infty}\frac{1}{t} \log\|\Lambda^i G_t(t,x)\|
\]
where $\Lambda^i G$ acts on $\Lambda^i \mathbf C^n$ by 
\[
\Lambda^i G(t,x) (u_1\wedge\ldots \wedge u_i)=G(t,x)u_1\wedge \ldots \wedge G(t,x) u_i\pp
\]
In particular, the greatest Lyapunov exponent is given by the norm of $G$ :
\[
\mu_n(x) = \lim_{t\to \pm \infty} \frac 1 t \log \|G(t,x)\|.
\]
If the matrix $G(t,x)$ is invertible for every $t$ then we also have 
\[
\mu_1(x) = \lim_{t\to \pm \infty} \frac 1 t \log \left( \left\|G(t,x)^{-1}\right\|\right)\pp
\]


\clearpage{\pagestyle{empty}\cleardoublepage}

\section*{Bibliography}

\addcontentsline{toc}{chapter}{Bibliographie}
\bibliographystyle{alpha}
\renewcommand{\refname}{\vspace{-0.4cm}}

\end{document}